\documentclass{amsart}
\title{Periodic points and tail lengths of split polynomial maps modulo primes}
\author{Benjamin Hutz}
\address{
Department of Mathematics and Statistics\\
Saint Louis University\\
St. Louis, MO}
\email{benjamin.hutz@slu.edu}

\author{Teerth Patel}
\address{
Saint Louis University\\
St. Louis, MO}
\email{pateltj@slu.edu}

\usepackage{color}
\usepackage{amssymb,amsmath,amsthm,fullpage}
\usepackage[all]{xy}  
\usepackage{url}
\usepackage{rotating} 
\usepackage{tikz}
\usepackage{textcomp,listings}
\usepackage{versions}
\usepackage{rotating} 
\usepackage{algorithm, algpseudocode}
\algrenewcommand\algorithmicrequire{\textbf{Input:}}

\definecolor{green}{rgb}{0,0.5,0}
\definecolor{dkgreen}{rgb}{0,0.6,0}
\definecolor{gray}{rgb}{0.5,0.5,0.5}
\definecolor{mauve}{rgb}{0.58,0,0.82}
\lstnewenvironment{python}[1][]{
\lstset{
  frame=single,                   
  language=python,                          
  basicstyle=\ttfamily\small, 
  numbers=left,                             
  numberstyle=\scriptsize\color{black},  
  stepnumber=1,                   
  numbersep=7pt,                  
  backgroundcolor=\color{white},      
  stringstyle=\color{red},
  showspaces=false,               
  showstringspaces=false,         
  showtabs=false,                 
  tabsize=2,                      
  captionpos=b,                   
  breaklines=true,                
  breakatwhitespace=false,        
  rulecolor=\color{black},        
  framexleftmargin=1mm, framextopmargin=1mm, frame=shadowbox, 
  alsoletter={1234567890},
  otherkeywords={\ , \}, \{},
  keywordstyle=\color{blue},       
  stringstyle=\color{mauve},         
  emph={access,and,break,class,continue,def,del,elif ,else,%
  except,exec,finally,for,from,global,if,import,in,i s,%
  lambda,not,or,pass,print,raise,return,try,while},
  emphstyle=\color{black}\bfseries,
  emph={[2]True, False, None, self},
  emphstyle=[2]\color{blue},
  emph={[3]from, import, as},
  emphstyle=[3]\color{blue},
  upquote=true,
  morecomment=[s]{"""}{"""},
  commentstyle=\color{dkgreen},       
  emph={[4]1, 2, 3, 4, 5, 6, 7, 8, 9, 0},
  emphstyle=[4]\color{blue},
  literate=*{:}{{\textcolor{blue}:}}{1}%
  {=}{{\textcolor{blue}=}}{1}%
  {-}{{\textcolor{blue}-}}{1}%
  {+}{{\textcolor{blue}+}}{1}%
  {*}{{\textcolor{blue}*}}{1}%
  {!}{{\textcolor{blue}!}}{1}%
  {(}{{\textcolor{blue}(}}{1}%
  {)}{{\textcolor{blue})}}{1}%
  {[}{{\textcolor{blue}[}}{1}%
  {]}{{\textcolor{blue}]}}{1}%
  {<}{{\textcolor{blue}<}}{1}%
  {>}{{\textcolor{blue}>}}{1},%
}}{}

\definecolor{orange}{rgb}{1,0.65,0.17}

\providecommand{\open}[1]{\textbf{Open:}\textcolor{red}{#1}}

\def\Z{\mathbb{Z}}

\def\P{\mathbb{P}}

\def\A{\mathbb{A}}

\def\F{\mathbb{F}}

\DeclareMathOperator{\Per}{Per}
\DeclareMathOperator{\Pre}{Pre}

 \DeclareMathOperator{\lcm}{lcm}
 
\DeclareMathOperator{\PGL}{PGL} 
 \DeclareMathOperator{\End}{End}

\DeclareMathOperator{\ord}{ord}

\theoremstyle{plain}
\newtheorem{thm}{Theorem}[section]
\newtheorem*{thm*}{Theorem}
\newtheorem{lem}[thm]{Lemma}
\newtheorem{prop}[thm]{Proposition}
\newtheorem{cor}[thm]{Corollary}
\theoremstyle{definition}
\newtheorem{defn}[thm]{Definition}

\newtheorem{exmp}[thm]{Example}
\newtheorem*{exmp*}{Example}

\theoremstyle{remark}
\newtheorem*{rem}{Remark}


\excludeversion{code}

\begin{document}
\maketitle


\begin{abstract}
    Explicit formulas are obtained for the number of periodic points and maximum tail length of split polynomial maps over finite fields for affine and projective space. This work includes a detailed analysis of the structure of the directed graph for Chebyshev polynomials of non-prime degree in dimension 1 and the powering map in any dimension. The results are applied to an algorithm for determining the type of a given map through analysis of its cycle statistics modulo primes.
\end{abstract}

\section{Introduction}

Given a finite set $S$ and a self-map $f:S \to S$, we can iterate $f$ as $f^n = f\circ f^{n-1}$ for $n \geq 1$ with $f^0$ defined as the identity map. We can define a directed graph $G_f$ whose vertices are given by the elements of $S$ and whose edges are $(x,f(x))$ for each $x \in S$. There are numerous questions one may ask about these graphs, such as what is the overall structure, the number of periodic versus non-periodic points, or the number of connected components. If one assumes that $f$ is a random mapping, where we define random mapping as the image of any given $x \in S$ is equally likely to be any element of $S$, then the statistics have been well studied because of their connections with computational number theory and cryptography; see, for example, the survey \cite{Mutafchiev2}. However, in practice, one typically uses explicit functions. Hence, the question of which functions behave as random maps is also well studied in certain instances. For quadratic polynomials $f(x) = x^2+c$, Pollard, in his $\rho$-factoring algorithm \cite{Pollard}, advised not to use $x^2$ or $x^2-2$ due to their non-random behavior. The statistics for these two maps, and some of their generalizations, have been extensively studied in dimension 1 \cite{Chou, Hernandez, Gilbert, Manes4, Shallit2}. It is well known that the only ``non-random'' polynomial maps in dimension 1 are those that come from an underlying group action; see Bridy for an explicit classification \cite{Bridy}.

In higher dimensions, the problem is much more complicated due to the additional freedom of multiple coordinates and interactions between the coordinate maps. This article resolves the statistics and structure of $G_f$ for higher dimensional maps constructed as each coordinate acted on independently by a polynomial map: a \emph{split-polynomial map}. Note that this includes the powering maps in all dimensions, but excludes maps such as the multivariate Chebyshev polynomials. This structure is easily computed in computer algebra systems for specific maps and for small primes. These easily computable quantities can then be compared to the structure theorems presented here to identify the map. Note that if the map is given in split form, this identification is clear from the defining polynomials of the map. However, the dynamical structure of a map is preserved by a change of variables through conjugation. After such a conjugation, the map can no longer be easily identified by examining its defining polynomials. Put more abstractly, this structure can be used to identify conjugacy classes in the moduli space of dynamical systems corresponding to split-polynomial maps.

These types of statistics have been studied by a number of authors in dimension 1 \cite{Chou, Hernandez, Gassert, Gilbert, Manes4, Shallit2}, and Striepel \cite{Streipel} does an empirical investigation of quadratic dynamical systems.
However, there seems to be few results in higher dimensions. Roberts and Vivaldi \cite{Vivaldi4} use cycle statistics modulo primes to classify maps of the affine plane as integrable, reversible, or neither. The first author applied their results to dynamical systems on a class of K3 surfaces in $\P^2 \times \P^2$ \cite{Hutz15}.

The main results and organization of this article is as follows.
Section \ref{sect_background} sets the basic definitions and notation. Section \ref{sect_dim1} summarizes the statistical results for random maps, gives the simple extension of the structure results for powering maps in dimension 1 to affine and projective spaces, and establishes the structure results for Cheybshev polynomials of non-prime degree. Section \ref{sect_dim1_alg} gives a simple algorithm for differentiating polynomial maps in dimension 1 using these structure theorems. Section \ref{sect_dimN} contains a count of the number of periodic points and a calculation of the max tail length for split polynomial maps in any dimension. Proposition \ref{prop_prod_count} provides the main tool for the main counting result in Theorem \ref{thm_dimN_periodic}. Section \ref{sect_dimN_alg} provides an algorithm for differentiating maps in dimension greater than one. Section \ref{sect_powering} provides a more detailed analysis of the powering map in higher dimensions. Finally, in Section \ref{sect_further} a few further directions of study are suggested.

\section{Background} \label{sect_background}
\subsection{Definitions and Notation}
Let $f:S \to S$ be a self-map of a set $S$. The \emph{minimal period} of a periodic point $x \in S$ is the smallest positive integer $n$ such that $f^n(x) = x$. A point $x$ is \emph{preperiodic} if there is a positive integer $m$ such that $f^m(x)$ is periodic. This $m$ is called the \emph{tail} of the preperiodic point. We associate to every preperiodic point a pair $(m,n) = (\text{tail}, \text{minimal period})$. We adopt the following notation:

\begin{align*}
    c(f,x) &= \text{ the minimal period of $x$ by $f$}\\
    t(f,x) &= \text{ the tail of $x$ by $f$}\\
    \Per_n(f,K) &= \{x \in K : c(f,x)= n\}\\
    \Per(f,K) &= \cup_{n\geq 1} \Per_n(f,K)\\
    \Pre_m(f,K) &= \{x \in K : t(f,x) = m\}\\
\end{align*}
    In this article, the set $S$ will be either the points of affine space over a finite field
    \begin{equation*}
        \A^N(\F_p) = \{(x_1,\ldots,x_N) : x_i \in \F_p\}
    \end{equation*}
    or the points of projective space over a finite field
    \begin{equation*}
        \P^N(\F_p) = \{(x_0,x_1,\ldots,x_N) : x_i \in \F_p\}/\sim
    \end{equation*}
    where $(x_0,\ldots,x_N) \sim (y_0,\ldots,y_N)$ if and only if
    \begin{equation*}
        x_iy_j = x_jy_i \quad \text{for all } 0 \leq i < j \leq N.
    \end{equation*}
    We use the notation $\#S$ to denote the number of elements in $S$.
    Note that many of the results in the literature use the set of points $\F_p^{\ast} = \A^1(\F_p) - \{0\}$.


    The self-mappings we consider are polynomial mappings.
    \begin{defn}
        A mapping $F: \P^1 \to \P^1$ is a \emph{polynomial mapping} if it has a totally ramified fixed point. Recall that a fixed point $z$ for $F$ is \emph{totally ramified} if $F^{-1}(z) = \{z\}$.
    \end{defn}
    If the totally ramified fixed point is the point at infinity, then the dehomogenization of $F$ ``looks'' like a polynomial in one variable. We make the equivalent definition in higher dimensions.
    \begin{defn}
        A mapping $F:\P^N \to \P^N$ is a \emph{polynomial mapping} if there is a totally invariant hyperplane.
    \end{defn}
    We are particularly interested in split polynomial mappings.
    \begin{defn}
        A \emph{split polynomial map} is a map of the form
        \begin{align*}
            f:\A^N &\to \A^N\\
            f(x_1,\ldots,x_{N}) &= (f_1(x_1),\ldots,f_{N}(x_{N}))
        \end{align*}
        for polynomials $f_i$. We will denote $f = f_1\times \cdots \times f_{N}$.

        Note, that a split polynomial map in dimension 1 is simply a polynomial map.
    \end{defn}
    Let $F:\P^N \to \P^N$ be the homogenization of a split polynomial map $f$. Notice that if the degrees of the $f_i$ are not all the same, then the resulting map is not a morphism in the dynamical system sense. In this case, reduction modulo primes does not commute with iteration, so we cannot identify $F$ by working modulo primes (see \emph{good reduction} in \cite[Chapter 2]{Silverman10}). Consequently, we will work only with split polynomial maps where the components are the same degree.

\section{Cycle Statistics in Dimension 1} \label{sect_dim1}
We give precise definitions for each type of map studied in dimension 1 and its cycle statistics. Polynomials maps in dimension 1 are thought to be either random or maps coming from an underlying group action: power maps or Chebyshev polynomials. Note that there are also automorphisms of the additive group $\mathbb{G}_a$, but they are linear maps $z \mapsto dz$, so have rather uninteresting dynamical properties.

For a more detailed background on maps associated to algebraic groups, see \cite[\S1.6]{Silverman10}.

\subsection{Random Maps}

    There are numerous results concerning random mapping statistics. The paper by Harris is one of the earliest appearing results \cite{Harris}. The survey from Mutafchiev \cite{Mutafchiev2}, while not recent, is also fairly comprehensive. The results on random maps most closely related to the current work is from Flajolet and Odlyzko \cite{Flajolet}, who prove the following theorem stated using our notation.
    \begin{thm}[\cite{Flajolet}] \label{thm_rnd_dim1}
        For a random mapping $f$ on a set $S$, we have the following asymptotic forms as $\#S \to \infty$.
        \begin{align*}
            \#\Per(f,S) &\sim \sqrt{\frac{\pi \cdot \#S}{8}} \quad (\text{\# of periodic pts})\\
            \bigcup_{0 \leq m < k} \Pre_m(f,S) &\sim (1-\tau_k)(\#S)\quad (\text{\# of $k$-th iterate image points})\\
            \max_x(t(f,x)) &= \log(2)\sqrt{2\pi\#S,} 
        \end{align*}
        where $\tau_k$ satisfies $\tau_0 = 0$ and $\tau_{k+1} = e^{-1+\tau_k}$.

    \end{thm}


\subsection{Power Maps}
    Consider the multiplicative group $\mathbb{G}_m(K) = K^{\ast}$. Its endomorphism ring is $\Z$:
    \begin{align*}
        \Z &\to \End(\mathbb{G}_m)\\
        d &\mapsto z^d.
    \end{align*}
    Hence, the powering maps are the endomorphism of an underlying group.

    \begin{exmp}
        The squaring map $P_2(x) =x^2$ is the basis of the Pepin primality test for Fermat numbers.
    \end{exmp}

    There is a number of results for the powering map in dimension 1; for example \cite{Chou,Lucheta,Sha,Vasiga, Wilson}. We summarize the results applicable to our problem in the following proposition. Recall that a vertex of a tree is a \emph{leaf} it is degree $1$. A vertex is degree $1$, in our case, if it has no preimages.
    \begin{prop} \label{prop_power}
        Let $p$ be a prime number and let $d>1$ be an integer. Let $P_d: \A^1 \to \A^1$ be the map $P_d(z) = z^d$. Let $m^{-}$ be the integer part of $p-1$ relatively prime to $d$ and $m^{+}$ be the integer part of $p+1$ relatively prime to $d$. Then
        \begin{align*}
            \#\Per(P_d,\A^1(\F_p)) &= m^-+1\\
            \#\Per(P_d,\P^1(\F_p)) &= m^-+2\\
            \max_x(t(P_d,x)) &= \max_q\left(\left\lceil \frac{v_q(p-1))}{v_q(d)}\right\rceil \right) \text{ for $q$ a prime divisor of $d$},
        \end{align*}
        where $v_q$ is the valuation with respect to $q$.
        Furthermore, at least $(p-1)/2$ points are leaves.
    \end{prop}
    \begin{proof}
        From \cite[Corollary 16]{Lucheta}, we know the number of periodic points of the powering map on $\F_p^{\ast}$ is $m^-$. For affine space, $0$ is fixed by $P_d$ so we get $m^-+1$. For projective space, the point at infinity is also fixed.

        The max tail length is proven by Ahmad-Syed \cite[Theorem 4.5]{Ahmad}.

        The number of leaves is proven by Lucheta \cite[Corollary 23]{Lucheta}.
    \end{proof}
%
%
%
%
%
%
%
%
%

\subsection{Chebyshev Polynomials}
    The multiplicative group $\mathbb{G}_m$ has a nontrivial automorphism $z \mapsto z^{-1}$. Taking the quotient of $\mathbb{G}_m$ by this automorphism, we have an isomorphism with affine space
    \begin{align*}
        \mathbb{G}_m/\{z \sim z^{-1}\} &\xrightarrow{\sim} \A^1\\
        z &\mapsto z + z^{-1}.
    \end{align*}
    Since the automorphism $z \mapsto z^{-1}$ commutes with the power map $P_d(z) = z^d$, the polynomial $P_d(z)$ induces a map on $\A^1$ that satisfies
    \begin{equation*}
        T_d(z+z^{-1}) = z^d + z^{-d}.
    \end{equation*}
    The maps $T_d(z)$ can be shown to be monic polynomials of degree $d$, called the Chebyshev polynomials (of the first kind). Hence, the Chebyshev polynomials are the endomorphism of an underlying group.

    \begin{exmp}
        The properties of the first Chebyshev polynomial $T_2(x) = x^2-2$ form the basis for the Lucas-Lehmer primality test for Mersenne numbers.
    \end{exmp}

    Gassert \cite{Gassert} studied $T_q$ over $\F_{p^n}$ for $q$ prime. However, for Sections \ref{sect_dim1_alg} and \ref{sect_dimN_alg}, we need the number of periodic points of $T_d$ for composite $d$, so now prove the appropriate generalizations.

    The key property is that Chebyshev polynomials commute, i.e., for any positive integers $a,b$
    \begin{equation*}
        T_a \circ T_b = T_b \circ T_a = T_{ab}.
    \end{equation*}

    \begin{lem}\label{lem_cheby_1}
         Let $a,b$ be positive integers. A point is periodic for $T_a \circ T_b$ if and only if it is periodic for both $T_a$ and $T_b$ (not necessary of the same period).
    \end{lem}
    \begin{proof}
        One direction is trivial. If a point $z$ is periodic for both $T_a$ and $T_b$ with minimal periods $m,n$, respectively, then
        \begin{equation*}
            T_{ab}^{\lcm(m,n)}(z) = T_a^{\lcm(m,n)}(T_b^{\lcm(m,n)}(z)) = T_a^{\lcm(m,n)}(z) = z.
        \end{equation*}

        Now assume there is a point $z$ that is periodic for $T_{ab}$ with minimal period $n$. Assume that $z$ is not periodic for $T_a$. Then
        \begin{equation*}
            T_{ab}^{kn}(z) = z
        \end{equation*}
        for all positive integers $k$ and
        \begin{equation*}
            T_{ab}^{kn}(z) = T_b^{kn}(T_a^{kn}(z)) = z.
        \end{equation*}
        In particular, the set $\cup_{k \geq 1}T_b^{-kn}(z)$ is infinite. Since our field is finite, this is a contradiction. Hence, $z$ is periodic for $T_a$.

        Reversing the order of $T_a$ and $T_b$, we make the same argument to see that $z$ is also periodic for $T_b$.
    \end{proof}
    \begin{thm} \label{thm_cheby_dim1_count}
        Let $p$ be a prime number and let $d>1$ be an integer. Let $T_d: \A^1 \to \A^1$ be the $d$th Chebyshev polynomial of the first kind. Let $m^{-}$ be the integer part of $p-1$ relatively prime to $d$ and $m^{+}$ be the integer part of $p+1$ relatively prime to $d$.
        The number of periodic points for $T_d$ is
        \begin{equation*}
            \Per(T_d,\A^1(\F_p)) = \frac{m^- + m^+}{2}
        \end{equation*}
        and
        \begin{equation*}
            \Per(T_d,\P^1(\F_p)) = \frac{m^- + m^+}{2} +1.
        \end{equation*}
    \end{thm}

    \begin{proof}
        We consider the $T_q$ for the prime divisors $q$ of $d$.

        We know from Gassert \cite[Theorem 2.4]{Gassert}, that the number of periodic points of $T_q$ is determined by summing over divisors of $m_q^-$ and $m_q^+$, where $m_q^-$ the integer part of $p-1$ relatively prime to $q$ and $m_q^+$ is the integer part of $p+1$ relatively prime to $q$. In particular,
        \begin{equation*}
            \#\Per(T_q)(\A^1(\F_p)) =
            \begin{cases}
                1 + \sum_{k \mid m_q^-, k>2}\frac{\varphi(k)}{2} + \sum_{k \mid m_q^+, k>2} \frac{\varphi(k)}{2} & q=2\\
                2+ \sum_{k \mid m_q^-, k>2}\frac{\varphi(k)}{2} + \sum_{k \mid m_q^+, k>2} \frac{\varphi(k)}{2} & q \text{ odd,}
            \end{cases}
        \end{equation*}
        where $\varphi$ is the Euler-phi function.
        Writing $T_d$ as the composition
        \begin{equation*}
            T_d = T_{q_1} \circ \cdots \circ T_{q_r}
        \end{equation*}
        and applying Lemma \ref{lem_cheby_1}, we find that the periodic points of $T_d$ must be periodic for each prime $q_i$. In particular, we need the $k$ which divide $m_{q_i}$ for $1 \leq i \leq r$. Recalling that $m^{\pm}$ are defined as the integer part of $p \pm 1$ relatively prime to $d$, we have the total number of periodic points as
        \begin{equation*}
            \#\Per(T_d)(\A^1(\F_p)) =
            \begin{cases}
                1 + \sum_{k \mid m^-, k>2}\frac{\varphi(k)}{2} + \sum_{k \mid m^+, k>2} \frac{\varphi(k)}{2} & d \text{ even}\\
                2+ \sum_{k \mid m^-, k>2}\frac{\varphi(k)}{2} + \sum_{k \mid m^+, k>2} \frac{\varphi(k)}{2} & d \text{ odd}.
            \end{cases}
        \end{equation*}
        In particular, we have
        \begin{equation*}
            \#\Per(T_d)(\A^1(\F_p)) =
            \begin{cases}
                1 + \frac{m^- -1}{2} + \frac{m^+-1}{2} = \frac{m^-+m^+}{2} & d \text{ even}\\
                2 + \frac{m^- -2}{2} + \frac{m^+-2}{2} = \frac{m^-+m^+}{2} & d \text{ odd}.
            \end{cases}
        \end{equation*}

        For the projective count, we also have the fixed point at infinity.
    \end{proof}
    \begin{rem}
        Notice that if $T_d$ is a permutation, then $P_d$ must be a permutation.
    \end{rem}
    \begin{rem}
        In the case $\gcd(k,p-1)=\gcd(k,p+1)=1$, which will occur infinitely often when $k$ is prime, we have
        \begin{equation*}
            \frac{m^- + m^+}{2} = \frac{(p-1) + (p+1)}{2} = p
        \end{equation*}
        so that the map is a permutation.
    \end{rem}

    Similarly, we generalize Gassert's results for tails.
    \begin{lem}\label{lem_cheby_2}
        Let $a,b$ be positive integers. Then
        \begin{equation*}
            t(T_{ab}(z)) = \max(t(T_a,z),t(T_b,z)).
        \end{equation*}
    \end{lem}
    \begin{proof}
        Let $w = \max(t(T_a,z),t(T_b,z))$. Then $T_a^w(z)$ and $T_b^w(z)$ are both periodic. By Lemma \ref{lem_cheby_1} $z$ is periodic for $T_{ab}^w$. Therefore, $t(T_{ab},z) \leq \max(t(T_a,z),t(T_b,z))$.

        Similarly, for any $w' < w$, $T_a^{w'}(z)$ or $T_b^{w'}(z)$ is not periodic, by Lemma \ref{lem_cheby_1}, $z$ is not periodic for $T_{ab}^{w'}$. Therefore, $t(T_{ab},z) \geq \max(t(T_a,z),t(T_b,z))$.
    \end{proof}
    \begin{prop}
        Let $d>1$ be a positive integer. Let $T_d: \A^1 \to \A^1$ be the $d$th Chebyshev polynomial of the first kind. Let $m^{-}$ be the integer part of $p-1$ relatively prime to $d$ and $m^{+}$ be the integer part of $p+1$ relatively prime to $d$. Then,
        \begin{equation*}
           \max_x(t(T_d,x)) =\max_q\left(\max\left(
               \left\lceil \frac{v_q(p-1))}{v_q(d)}\right\rceil\right), \max\left(
               \left\lceil \frac{v_q(p+1))}{v_q(d)}\right\rceil\right)
               \right)
        \end{equation*}
        where $q$ ranges over the prime divisors of $d$. 
    \end{prop}
    \begin{proof}
        Apply Gassert \cite[Theorem 2.3]{Gassert} to each $T_q$ for $q$ dividing $d$ and combine with Lemma \ref{lem_cheby_2}.
    \end{proof}

    Unlike for power maps, most Chebyshev polynomials of prime degree do not have at least half of the points as leaves.
    \begin{prop}
        Let $d$ be a prime. Let $T_d: \A^1 \to \A^1$ be the $d$th Chebyshev polynomial of the first kind. Then for $d=2$ and any prime $p$ or for pairs $(d,p)$ with $2d=p+1$ and $p$ prime, we have $\frac{p-1}{2}$ leaves.
    \end{prop}
    \begin{proof}
        Since we are considering only prime $d$, we can appeal directly to the counting in Gassert \cite[Theorems 2.3, 2.4]{Gassert}. For $d=2$, the number of points at height $k$ is given in Gassert as
        \begin{equation*}
            \begin{cases}
              2^{k-2}m^{\pm} & k \geq 2\\
              \frac{m^- + m^+}{2} & k=1
            \end{cases}
        \end{equation*}
        with the $\pm$ depending on whether $v_2(p-1)$ or $v_2(p+1)$ is larger. Let $\nu = \max(v_2(p-1), v_2(p+1)) \geq 2$. Since $\gcd(p-1,p+1) = 2$, we also have $\min(v_2(p-1), v_2(p+1))=1$. Then we have the number of leaves at height $\nu$ as
        \begin{equation*}
            2^{\nu-2}m^{\pm} = \frac{p\pm 1}{4}.
        \end{equation*}
        The rest of the leaves occur at height $1$ and there are $\frac{m^-+m^+}{2}$ of those. Of those, we need to know which are not leaves. All leaves map $2$-to-$1$ except for the cycle $0 \to -2 \to [2 \to 2]$. So we have the number of leaves as
        \begin{align*}
            \frac{p \pm 1}{4} + \frac{m^- + m^+}{2} - \frac{1}{2}\left(\frac{p \pm 1}{4} \cdot \frac{1}{2^{\nu-2}} -1 \right)
            &= \frac{p \pm 1}{4} + \frac{(p \pm 1)/2^\nu + (p \mp 1)/2}{2} - \frac{p \pm 1}{2^{\nu-1}}  -\frac{1}{2}\\
            &= \frac{p \pm 1}{4} + \frac{(p \pm 1)}{2^{\nu+1}} + \frac{p \mp 1}{4} - \frac{p \pm 1}{2^{\nu+1}}  -\frac{1}{2}\\
            &= \frac{p \pm 1}{4} + \frac{p \pm 1}{4} - \frac{1}{2}\\
            &= \frac{p-1}{2}.
        \end{align*}

        For $d  > 2$ we have the number of points at maximal height as
        \begin{equation*}
            (d-1)d^{\max(v_d(p-1),v_d(p+1))-1}\frac{m^{\pm}}{2} = (d-1) \frac{p\pm 1}{2d} \leq \frac{p-1}{2},
        \end{equation*}
        with equality only for $p+1 = 2d$. There are no leaves of smaller height since $d>2$ divides only one of $p\pm 1$.
    \end{proof}
    Note that by Lemma \ref{lem_cheby_1} the number of leaves can only increase for composite $d$. So we can have $T_d$ with at least half leaves for other composite choices of $(d,p)$.

%
%
%

\section{Differentiating maps in dimension 1} \label{sect_dim1_alg}
    On the set of degree $d$ self-maps on $\P^N$ there is a natural conjugation action by elements of $\PGL_{N+1}$. For $F:\P^N \to \P^N$ define
    \begin{equation*}
        F^{\alpha} = \alpha^{-1} \circ F \circ \alpha, \quad \alpha\in \PGL_{N+1}.
    \end{equation*}
    This action preserves the dynamical properties of $F$ since $(F^{\alpha})^n = (F^n)^{\alpha}$ and is the dynamical system equivalent of change of variables. It induces a similar action for self-maps of affine space $\A^N$. Consequently, while we may distinguish the powering map $x^n$ as special, there is actually an entire conjugacy class of maps that has the same dynamical properties as the powering map. For example,
    \begin{equation*}
        f(x) = \frac{x^2}{2x^2 - 2x + 1}
    \end{equation*}
    is conjugate to the powering map by the linear fractional transformation
    \begin{equation*}
        \alpha = \frac{1-x}{x} \in \PGL_2.
    \end{equation*}
    For applications, it is important to recognize when a given function is of certain type, such as a powering map or Chebyshev polynomial. Since there are only two special conjugacy classes of polynomial maps of each degree in dimension 1, we could use an algorithm from Faber-Manes-Viray \cite{FMV} to determine if a given map is conjugate to the appropriate powering or Chebyshev polynomial. While this algorithm is sufficient in dimensional 1, albeit quite slow for large degree, we will see in Section \ref{sect_dimN_alg} that it is not sufficient for higher dimensions because there are infinitely many distinct conjugacy classes with the same (asymptotic) cycle statistics.
    As an alternative, we use the results of Section \ref{sect_dim1} to formulate an algorithm for distinguishing the different polynomial maps in dimension $1$ through cycle statistics. We need to distinguish three separate cases:
    \begin{enumerate}
        \item polynomials that behave like random maps,
        \item the powering maps, and
        \item the Chebyshev polynomials.
    \end{enumerate}
    Theorem \ref{thm_rnd_dim1}, Proposition \ref{prop_power}, and Theorem \ref{thm_cheby_dim1_count} provide an order of growth for the number of periodic points of random maps and exact counts of periodic points for powering and Chebyshev polynomials. Consequently, by choosing a sequence of primes all of a certain form, we can distinguish between the types of maps. The following Lemma gives two examples of primes with particularly useful forms for $d$ even and odd, respectively.

    \begin{lem} \label{lem_dim1_prime_forms}
        Let $d \geq 2$ be an integer.
        \begin{enumerate}
            \item For Mersenne primes, $p=2^q-1$ and $d = 2^k$ a power of $2$, we have
                \begin{equation*}
                    m^+ = 1 \quad \text{and} \quad m^- = 2^{q-1}-1.
                \end{equation*}

            \item For Sophie Germain primes, $p=2q+1$ where $q$ is prime and $d$ is even, we have
                \begin{equation*}
                    m^- = q \quad \text{for}\quad q \nmid d.
                \end{equation*}


            \item For primes of the form $p \equiv -1 \pmod{d}$ and $d$ odd, we have
                \begin{equation*}
                    m^- = p-1.
                \end{equation*}
        \end{enumerate}
    \end{lem}
    \begin{proof}
        \mbox{}
        \begin{enumerate}
            \item Let $p = 2^q-1$ be a Mersenne prime and $d= 2^k$ a power of $2$. Then $p+1 = 2^q$ so that $m^+ = 1$ and $p-1 = 2(2^{q-1}-1)$ with $d$ even, so that $m^- = 2^{q-1}-1$.

            \item For $p= 2q+1$ a Sophie Germain prime and $d$ even, we have $p-1 = 2q$ so that $m^-= q$ when $q$ is relative prime to $d$. In other words, when $q$ is not one of the prime divisors of $d$, there are finitely many such exceptions.

            \item Let $d$ be odd and $p$ a prime with $p \equiv -1 \pmod{d}$. We have
                \begin{equation*}
                    p-1 \equiv d-2 \pmod{d}
                \end{equation*}
                so that $m^- = p-1$ since $d$ is odd. \qedhere
        \end{enumerate}
    \end{proof}
    \begin{rem}
        Note that $d=2^k$ is the only case where we have $m^{+} =1$ since this is the only case where $p+1 = d^k$ because $d-1 \mid (d^k-1)$. In the case $d>2$, we get $m^{+}$ is the maximum possible. So if $\gcd(d,2^{q-1}-1) = 1$, both $T_d$ and $C_d$ are permutations.
    \end{rem}

    \begin{algorithm}
    \caption{Differentiating polynomials in dimension 1}
    \label{alg_dim1}
    \begin{algorithmic}[1]
        \Require A self-map $f$ of degree $d$
        \If {$d$ is a power of $2$}
        \State Choose a sequence of Mersenne primes $\mathcal{P} = \{p_i\}$
        \ElsIf {$d$ is even}
        \State Choose a sequence of Sophie German primes $\mathcal{P} = \{p_i\}$
        \Else
        \State Choose a sequence of primes $\mathcal{P} = \{p_i\}$ such that $p_i\equiv -1 \pmod{d}$ for each $i$
        \EndIf
        \State Determine the sequence of values $N_i = \#\Per(f,\F_{p_i})$
        \If {the sequence $\{N_i\}$ is increasing as $\sqrt{p_i}$}
        \State Return ``random''
        \Else
            \State Compare the result to Proposition \ref{prop_power} and Theorem \ref{thm_cheby_dim1_count} and return either ``Power'' or ``Chebyshev''
        \EndIf
    \end{algorithmic}
    \end{algorithm}
    Because the order of growth of the number of periodic points is approximately $\sqrt{p}$ for random maps and $p$ for power and Chebyshev maps, distinguishing random from not random is determined by the rate of growth. Distinguishing between the two non-random cases of $T_d$ and $P_d$ is done through the explicit counts. Lemma \ref{lem_dim1_prime_forms} ensures that the chosen prime will give distinct values for the number of periodic points.

\section{Cycle Statistics in Dimension $>1$} \label{sect_dimN}

    In this section, we analyze split polynomial maps in dimension greater than $1$.
    The global dynamics of this type of map has been studied previously; see for example \cite{ghioca4, Medvedev, Nguyen}. We study the cycle statistics of these maps modulo primes.

\subsection{Total number of periodic points}

    The following proposition provides that main tool for studying periodic points of split polynomial maps.
    \begin{prop} \label{prop_prod_count}
        Let $f:\A^n\to \A^n$ and $g: \A^m \to \A^m$ be split polynomial maps. Let $h = f \times g$ be the product map. Then we have
        \begin{equation*}
            \#\Per(h,\A^{n+m}(\F_p)) = \#\Per(f,\A^n(\F_p))\cdot \#\Per(g,\A^m(\F_p)).
        \end{equation*}
        Furthermore, if the coordinate functions of $f$ and $g$ all have the same degree $d$, then for $H$, the homogenization of $h$, we have
        \begin{equation*}
            \#\Per(H,\P^{n+m}(\F_p)) = \#\Per(h,\A^{n+m}(\F_p)) + \#\Per(P_d,\P^{n+m-1}(\F_p)),
        \end{equation*}
        where $P_d$ is the $d$-th powering map.
    \end{prop}
    \begin{proof}
        Assume that $(x,y)$ is a periodic point for $h$ with $x \in \A^n$ and $y \in \A^m$. Then, since $h$ is a split polynomial map, $x$ is periodic for $f$ and $y$ is periodic for $g$. Let $\{x_1,x_2,\ldots,x_s\}$ be the cycle containing $x=x_1$ for $f$ and $(y_1,\ldots,y_t)$ be the cycle containing $y=y_1$ for $g$. Then each tuple $(x_i,y_j)$ for $1 \leq i \leq s$ and $1 \leq j \leq t$ determines a distinct periodic point for $h$. Hence,
        \begin{equation*}
            \Per(h,\A^{n+m}(\F_p)) = \Per(f,\A^n(\F_p))\cdot \Per(g,\A^m(\F_p)).
        \end{equation*}

        In the projective case, if the coordinates of $f$ and $g$ are polynomials of the same degree, then $H$ is a morphism. The periodic points in the affine chart with $x_{n+m} \neq 0$ are as in the previous part, so we need only consider the contribution from the points at infinity. Since the coordinate functions $f_i,g_j$ are all polynomials of the same degree, at $x_{n+m}=0$, we end up with the powering map $(x_0^d,\ldots,x_{n+m-1}^d,0)$ acting on the first $n+m$ coordinates. This is the same as the $d$-th power map on $\P^{n+m-1}$. Hence,
        \begin{equation*}
            \#\Per(H,\P^{n+m}(\F_p)) = \#\Per(h,\A^{n+m}(\F_p))) + \#\Per(P_d,\P^{n+m-1}(\F_p)).
            \qedhere
        \end{equation*}
    \end{proof}

    \begin{cor}\label{cor_power_periodic}
        Let $P_d^N$ be the powering map in dimension $N$ of degree $d$. Then,
        \begin{align*}
            \#\Per(P^N_d, \A^N(\F_p)) &= (m+1)^N\\
            \#\Per(P^N_d, \P^N(\F_p)) &= \sum_{i=0}^N (m+1)^i.
        \end{align*}
    \end{cor}

    \begin{thm} \label{thm_dimN_periodic}
        Define the following maps for $\A^1 \to \A^1$: $P_d$, the $d$th powering map and $T_d$, the $d$th Chebyshev polynomial of the first kind. Also define $R_k$ a random mapping $\A^k \to \A^k$.

        Consider the product map
        \begin{equation*}
             \phi= P_d^{\times a} \times T_d^{\times b} \times \prod_{i} R_{k_i}^{\times c_i}
        \end{equation*}
        where $f^{\times n} = f \times \cdots \times f$ $n$ times. Let $K = \sum k_i\cdot c_j$ and $N = a + b + K$.

        Let $p$ be a prime, $m^{-}$ be the integer part of $p-1$ relatively prime to $d$, and $m^{+}$ be the integer part of $p+1$ relatively prime to $d$.
        Then we have
        \begin{equation*}
            \#\Per(\phi,\A^N) =
            \begin{cases}
                (m^-+1)^a \cdot \left(\frac{m^-+m^+}{2}\right)^b & K =0\\
                O\left((m^-+1)^a \cdot \left(\frac{m^-+m^+}{2}\right)^b p^{K/2}\right) & K \neq 0.
            \end{cases}
        \end{equation*}
        where $O(\cdot)$ represents big-O notation.
        Let $\Phi$ be the homogenization of $\phi$. Then we have
        \begin{equation*}
            \#\Per(\phi,\P^N) =
            \begin{cases}
                (m^-+1)^a \cdot \left(\frac{m^-+m^+}{2}\right)^b + \sum_{i=0}^{N-1} (m^-+1)^i & K=0\\
                O\left((m^-+1)^a \cdot \left(\frac{m^-+m^+}{2}\right)^b p^{K/2}\right) & K \neq 0.
            \end{cases}
        \end{equation*}
    \end{thm}


    \begin{proof}
        The $K=0$ formula comes from Proposition \ref{prop_prod_count} combined with the counts from Section \ref{sect_dim1}.

        For $K \neq 0$, we use Proposition \ref{prop_prod_count} with Section \ref{sect_dim1} for the group derived portion and Theorem \ref{thm_rnd_dim1} for the big-$O$ of the random portion.
    \end{proof}
    \begin{exmp}
        Consider the map
        \begin{align*}
            \phi = P_2 \times (T_2)^{\times 2}:\A^3 &\to \A^3\\
            (x,y,z) &\mapsto (x^2, y^2-2, z^2-2).
        \end{align*}
        Then for $p=37$, we compute $m^- = 9$ and $m^+ = 19$ so that
        \begin{align*}
            \#\Per(\phi,\A^3(\F_p)) &= 10 \cdot \left(\frac{9+19}{2}\right)^2 = 1960\\
            \#\Per(\phi,\P^3(\F_p)) &= 1960 + \sum_{i=0}^{2}(10)^i = 2071.
        \end{align*}
\begin{code}
\begin{python}
P.<x,y,z,w>=ProjectiveSpace(GF(37),3)
H=End(P)
Td=H([x^2,y^2-2*w^2,z^2-2*w^2,w^2])
G=Td.cyclegraph()
sum([len(t)-1 for t in G.all_simple_cycles()])
\begin{python}
\end{code}
    \end{exmp}

\subsection{Tail lengths}
    The following lemma provides the key tool to studying tails for split polynomial maps. Define the \emph{$k$th preimages} of a point $x$ under the map $f$ as
    \begin{equation*}
        f^{-k}(x) = \{z : f^k(z) = x\}.
    \end{equation*}
    \begin{lem}
         Let $f:\A^n\to \A^n$ and $g: \A^m \to \A^m$ be split polynomial maps. Let $h = f \times g$ be the product map. Let $(x,y) \in \A^{n+m}$. For any positive integer $k$, we have the equality
        \begin{equation*}
            h^{-k}(x,y) = \{(u,v) : f^k(u) = x, g^k(v) = y\}.
        \end{equation*}
    \end{lem}
    \begin{proof}
        Since the maps involved are split polynomial maps, each coordinate is independent.
    \end{proof}
    \begin{cor}
        Let $f:\A^n\to \A^n$ and $g: \A^m \to \A^m$ be split polynomial maps. Let $h = f \times g$ be the product map. Then for any point $(x,y) \in \A^{n+m}$ we have
        \begin{align*}
            \#h^{-k}(x,y) = (\#f^{-k}(x))(\#g^{-k}(y)).
        \end{align*}
        Note that these sets include periodic points.
%\open{Maybe we get use this to get the number of strictly periodic points at some level by subtracting off the periodic ones???}
    \end{cor}
    \begin{cor}
         Let $f:\A^n\to \A^n$ and $g: \A^m \to \A^m$ be two split polynomial maps. Let $h = f \times g$ be the product map. Then
         \begin{equation*}
            t(h,(x,y)) = \max(t(f,x),t(g,y)).
         \end{equation*}
    \end{cor}

    \begin{prop} \label{prop_cheby_tail}
        Define the following maps $\A^1 \to \A^1$, $P_d$ the $d$th powering map, $T_d$, the $d$th Chebyshev polynomial of the first kind. Also define $R_k$ a random mapping $\A^k \to \A^k$.

        Consider the product map
        \begin{equation*}
             \phi= P_d^{\times a} \times T_d^{\times b} \times \prod_{i} R_{k_i}^{\times c_i}
        \end{equation*}
        where $f^{\times n} = f \times \cdots \times f$ $n$ times. Let $K = \sum k_i\cdot c_j$ and $N = a + b + K$.

        Then we have, where $q$ ranges over prime divisors of $d$,
        \begin{equation*}
            \max_x(t(\phi,x)) =
            \begin{cases}
               \max_q\left(\left\lceil \frac{v_q(p-1))}{v_q(d)}\right\rceil \right)  & K =0, b=0\\
               \max_q\left(\max\left(
               \left\lceil \frac{v_q(p-1))}{v_q(d)}\right\rceil\right), \max\left(
               \left\lceil \frac{v_q(p+1))}{v_q(d)}\right\rceil\right)
               \right)  & K =0, b\neq 0\\
                O(p^{\max(k_i)/2}) & K \neq 0.
            \end{cases}
        \end{equation*}
    \end{prop}
    \begin{proof}
    We combine Theorem \ref{thm_rnd_dim1}, Proposition \ref{prop_power}, and Proposition \ref{prop_cheby_tail}.
    \end{proof}
    \begin{rem}
        The proposition provides the key to distinguishing the product of two random maps from a truly random map. For example, the product of two random dimension 1 maps will have largest tail on the order of $\sqrt{p}$, whereas a random map in dimension 2 will have largest tail on the order of $p$.
    \end{rem}

\section{Differentiating maps in dimension greater than 1} \label{sect_dimN_alg}

    We use Lemma \ref{lem_dim1_prime_forms}, Theorem \ref{thm_dimN_periodic}, and Proposition \ref{prop_cheby_tail} to differentiate between split polynomial maps through cycle statistics. Note that if we are given a split polynomial map in split form (i.e., each of the defining polynomials is a single variable polynomial), then we could simply apply Algorithm \ref{alg_dim1} to each coordinate separately. However, conjugation will not change the dynamics, but will result in a ``mixing'' of the coordinates. Additionally, if there is a component that is random on dimension $k$ (as opposed to the product of $k$ random dimension 1 maps), then it will be not split into $k$ separated polynomials.

    The idea of the general algorithm is as follows. We can determine the total amount of randomness by looking at the growth of the number of periodic points. Using growth of the tails, we can determine the highest dimension of the random components but not the dimensions of each random component. If there is no random component, then exact counts of periodic points determine the map. If there is a random component, the choice of prime sequence will cause the number of periodic points of the different types of maps to diverge. The choice of this sequence of primes is somewhat delicate. The ideal sequence is a sequence where $m^{-}$ is a constant over all values of the sequence, such as the Mersenne primes for $d = 2^k$. The clear drawback of using Mersenne primes is that the explicit count of periodic points grows exponentially in the prime (with exponent depending on the dimension). Consequently, large primes are not feasible. Fortunately, in practice it is sufficient to find a short sequence of primes that are not too large for which $m^{-}$ is constant. For example, the sequence $\{7, 13, 87, 173, 793\}$ has $m^- = 3$ when $d=4$.

    \begin{algorithm}
    \caption{Differentiating polynomials in dimension greater than 1}
    \label{alg_dimN}
    \begin{algorithmic}[1]
        \Require A self-map $f$ of degree $d$ on dimension $N$
        \State Choose a sequence of primes $\{p_i\}$ for which $m^-$ is constant.
        \State Determine the sequence of values $N_i = \#\Per(f,\F_{p_i})$.
        \State Graph the $N_i$ versus $p$ on the same graph as a representative curve for each of the possible split polynomial combinations.
        \State Determine the type that most closely matches the graphs.
        \State To determine the size of the random component, plot the largest tail length. Correlate to the size of the largest random component.
        \State Return the type that most closely matches the graphs.
    \end{algorithmic}
    \end{algorithm}


\section{The Powering Map on $\A^N$} \label{sect_powering}
    While it is possible to treat the powering map of degree $d$ on $\A^N$ as a split polynomial map and apply the results of Section \ref{sect_dimN}, it is also possible to generalize the analysis from dimension 1 as done in the literature \cite{Chou, Lucheta, Sha, Vasiga, Wilson}. We generalize Proposition \ref{prop_power} for any dimension. In the case where $d$ is prime, we also describe the precise tree structure of $G_{P_d}$ and count the number of preperiodic points with given tail length.

    Let $\ord_p$ be the multiplicative order modulo $p$ and $v_p(x)$ the valuation with respect to $p$.
\begin{code}
\begin{python}
p=13
P.<x,y>=ProjectiveSpace(GF(p),1)
H=End(P)
d=2
pd = d.prime_divisors()
m=(p-1)/prod([(q**((p-1).ord(q))) for q in pd])
f=H([x^d,y^d])
for Q in P:
    if Q[0]!=0:
        Q.orbit_structure(f), max([val(ord(Q[0],p),q)/val(d,q) for q in pd]), ord(d,gcd(ord(Q[0],p),m))
\end{python}
\end{code}
    \begin{lem}
        Let $p$ be a prime and $d>1$ an integer. Let $P_d:\A^N \to \A^N$ be the degree $d$ powering map. Let $m^-$ be the integer part of $p-1$ relatively prime to $d$.
        \begin{align*}
            t(F,z) &= \max_{i, z_i \neq 0}\left(\max_q\left(\left\lceil \frac{v_q(\ord_p(z_i))}{v_q(d)}\right\rceil\right)\right) \quad \text{ $q$ a prime factor of $d$, } 1 \leq i \leq N\\
            c(F,z) &= \lcm(\ord_{\gcd(\ord_p{z_i},m^{-})},d)), \quad 1 \leq i \leq N, z_i \neq 0.
        \end{align*}
    \end{lem}
    \begin{proof}
        In dimension $1$ from Sha \cite[Prop 3.1]{Sha}, we have
        \begin{align*}
            t(f,z) &= \max_q\left(\left\lceil \frac{v_q(\ord_p(z))}{v_q(d)} \right\rceil\right) \text{ $q$ a prime factor of $d$}\\
            c(f,z) &= \ord_{\gcd(\ord_p{z},m^{-})},d).
        \end{align*}
        To conclude the lemma, we note that since the map is split, the minimal period of a point is the least common multiple of the minimal periods of the coordinates under the single variable powering map of degree $d$. Similarly, the tail of $(z_1,\ldots,z_N)$ is the max of the tails of the coordinates individually.
    \end{proof}
    We prove the following counts of periodic points.
    \begin{thm}\label{thm_powering-dimN}
        Let $p$ be prime and $d>1$ an integer. Let $P^N_d$ be the degree $d$ powering map. Let $m^-$ be the integer part of $p-1$ relatively prime to $d$.
        Fix a positive integer $k$; then
        \begin{align*}
            \#\Per_k(P_d^N, \P^N)&=\sum_{D=0}^N \left( \delta_k + \sum_{I=1}^D \mathop{\sum_{k_i \mid m^{-}}}_{\lcm(\ord_d(k_i)) = k} \binom{D}{I}\prod_{i=1}^I \varphi(k_i)\right)\\
            \#\Per_k(P_d^N, \A^N)&=\delta_k + \sum_{I=1}^N \mathop{\sum_{k_i \mid m^{-}}}_{\lcm(\ord_d(k_i)) = k} \binom{N}{I}\prod_{i=1}^I \varphi(k_i)
        \end{align*}
        where $\delta_k = 1$ if $k=1$ and $0$ otherwise.
        Furthermore,
        \begin{align*}
            \#\Per(P_d^N, \P^N) &= \sum_{D=0}^N  1+\sum_{I=1}^D\left(\mathop{\sum_{k_i \mid m^-}}_{1 \leq i \leq I} \binom{D}{I} \left( \prod_{i=1}^I \varphi(k_i) \right)\right)= \sum_{D=0}^N(m^{-}+1)^D\\
            \#\Per(P_d^N, \A^N) &=  1+\sum_{I=1}^N\left(\mathop{\sum_{k_i \mid m^-}}_{1 \leq i \leq I} \binom{N}{I} \left( \prod_{i=1}^I \varphi(k_i) \right)\right)= (m^{-}+1)^N.
        \end{align*}
    \end{thm}
    \begin{proof}
        For each positive divisor $k$ of $m^-$, the powering map on $\A^1$ contains $\frac{\varphi(k)}{\ord_k(d)}$ cycles of length $\ord_k(d)$ \cite[Theorem 1]{Chou} for $k\neq 1$ and one additional cycle ($z=0$) when $k=1$. For counting the number of periodic points, we need to sum over the divisors $k_i$ of $m^-$ for each coordinate. The period is the least common multiple of the periods of the coordinates. For each choice of $k_i$ we have $\frac{\varphi(k_i)}{\ord_d(k_i)}$ cycles of length $\ord_d(k_i)$ in that coordinate. That makes $\gcd(\ord_d(k_i))$ total cycles of length $\lcm(\ord_d(k_i))$. To count those of period exactly $k$, we sum over those combinations of $k_i$ with $\lcm(\ord_d(k_i))=k$. These counting formula are only good for nonzero coordinates, so we sum over the number of nonzero coordinates as well (and the permutations based on the number of zeros).

        For minimal period exactly $k$, there is little simplification to be done when counting this way. However, for all periodic points, we arrive at
        \begin{align*}
            \#\Per(P_d^N) &= \sum_{D=0}^N 1+ \sum_{I=1}^D \left(\mathop{\sum_{k_i \mid m^-}}_{1 \leq i \leq I} \binom{D}{I} \left( \lcm(\ord_d(k_i))\gcd(\ord_d(k_i)) \prod_{i=1}^I \frac{\varphi(k_i)}{\ord_d(k_i)} \right)\right)\\
            &= \sum_{D=0}^N 1+ \sum_{I=1}^D \left(\mathop{\sum_{k_i \mid m^-}}_{1 \leq i \leq I} \binom{D}{I} \left( \prod_{i=1}^I \varphi(k_i) \right)\right)\\
            &= \sum_{D=0}^N 1 + \sum_{I=1}^D \binom{D}{I} \sum_{k_1 \mid m^-}\varphi(k_1)\left( \cdots \left(\sum_{k_I \mid m^-} \varphi(k_I) \right)\cdots \right)\\
            &= \sum_{D=0}^N 1 + \sum_{I=1}^D \binom{D}{I} (m^-)^I
            = \sum_{D=0}^N (m^{-}+1)^D.
        \end{align*}
    \end{proof}

    Note that we have arrived at the conclusion of Corollary \ref{cor_power_periodic} by use of the explicit formulas.

    We now describe the tails and tree structure. First note that for $p=2$, every power map is a permutation comprised entirely of fixed points, so we exclude the case of $p=2$ from consideration. Also note that if the degree of the map is relatively prime to $p-1$, then the map is also a permutation.

    \begin{thm}
        Let $p$ be a prime and $d>1$ an integer. Let $P^N_d$ be the degree $d$ powering map with $N>1$. Over both $\P^N$ and $\A^N$, the max tail is $\ord_d(p-1)$. If $\gcd(d,p-1) \neq 1$, then at least half of all points are leaves:
        \begin{align*}
            \#\{\text{leaves for } \A^N\} &> \frac{p^N}{2} = \frac{\#\A^N(\F_p)}{2}\\
            \#\{\text{leaves for } \P^N\} &> \sum_{D=0}^N \frac{p^D}{2} = \frac{\#\P^N(\F_p)}{2}.
        \end{align*}
    \end{thm}
    \begin{proof}
        Lucheta \cite[Theorem 31]{Lucheta} gives the max tail in dimension $1$ as $\ord_d(p-1)$. Since the tail in dimension $N$ is the max of the tails of the coordinates, we have the max tail is still $\ord_d(p-1)$.

        We know from Lucheta \cite[Corollary 23]{Lucheta} that for $P_d$ for $\A^1$ at least $(p-1)/2$ points are leaves. To be a leaf in dimension $N$, a point must be a leaf for at least one coordinate as a dimension 1 point. In particular, if one coordinate is a leaf, then the other $N-1$ coordinates can be chosen arbitrarily. Applying this to each coordinate, we count the total number of leaves for $\A^N(\F_p)$ as
        \begin{equation*}
            N\frac{p-1}{2}p^{N-1} = \frac{N}{2}(p^N - p^{N-1}) \geq p^N - p^{N-1} = p^{N-1}(p-1) >  p^{N-1}\frac{p}{2} = \frac{p^N}{2}.
        \end{equation*}

        For projective space, we have a similar count on each set of coordinates with $z_D = 1$ and $z_i =0$ for $D < i \leq N$.         Since the point $(1,0\ldots,0)$ is a totally ramified fixed point, it is not a leaf, so we start at $D=1$.
        \begin{align*}
            \sum_{D=1}^N D\frac{p-1}{2}p^{D-1}
            &= \frac{p-1}{2} + 2\frac{p-1}{2}p + \cdots + N\frac{p-1}{2}p^{N-1}\\
            &> \frac{p-1}{2} + \frac{p-1}{2}p + \cdots + \frac{p-1}{2}p^{N-1} + 1 \quad \text{since $N \geq 2$}\\
            &> \frac{1+p}{2} + \frac{p^2}{2} + \cdots + \frac{p^D}{2} \quad \text{since $p>2$.}\qedhere
        \end{align*}
    \end{proof}

    For $d$ a prime, we can describe the full tree structure.
    \begin{prop}
        Let $p$ be a prime and $d>1$ be an integer. Let $P_d$ be the degree $d$ powering map on dimension $N>1$. Every vertex has either $0,1,d,\ldots,d^N$ preimages. If $d$ is prime, then the trees are as full as possible of level $v_d(p-1)$, i.e., the only missing preimages are those of coordinates $0$.
    \end{prop}
    \begin{proof}
        For the preimages of a given point, we are solving $x^d = z$ for each coordinate $z$. There can be no solutions if $z$ is not a $d$-th power residue, $1$ solution if $z$ is $0$, and $d$ solutions otherwise. Hence, the total number of preimages depends on the number of nonzero coordinates when all coordinates are $d$th power residues.

        When $d$ is prime, Lucheta \cite[Section 7]{Lucheta} proves the trees are full trees for the non-zero points. In the higher dimensional case, we get full preimage sets of each nonzero coordinate and the single primage 0 for the 0 coordinates.
    \end{proof}
    \begin{cor}
        Let $p,d$ be prime numbers. Let $P_d$ be the degree $d$ powering map in dimension $N$. Let $m^{-}$ be the integer part of $p-1$ relatively prime to $d$. We count the number of strictly preperiodic points for $P_d$ as
        \begin{align*}
            \sum_{k>0}\#\Pre_k(P_d,\A^N(\F_p)) &= p^N - (m^-+1)^N\\
            \sum_{k > 0} \#\Pre_k(P_d,\P^N(\F_p)) &= \sum_{D=0}^N p^D - \sum_{D=0}^N (m^-+1)^D
        \end{align*}
        with the number with tail $k$ as
        \begin{align*}
            \#\Pre_k(P_d,\A^N(\F_p)) = \sum_{e=0}^N(d^e-1)\binom{N}{e} (m^-)^{N-e}(d^{N-e})^k\\
            \#\Pre_k(P_d,\P^N(\F_p)) = \sum_{D=0}^N\sum_{e=0}^D(d^e-1)\binom{D}{e} (m^-)^{D-e}(d^{D-e})^k.
        \end{align*}
    \end{cor}
    \begin{proof}
        The strictly preperiodic points are full trees of level $\nu=v_d(p-1)$. We split the sum into pieces based on how many coordinates are $0$; e.g., if there are $e$ zero coordinates, then the first set of preimages has $d^{N-e}$ points. This gives the stated formula since at the first level we are missing one node (for the periodic point). For projective space we do this for each sub-dimension ($z_D = 1$ and $z_i = 0$ for $D<i\leq N$).

        We then get the total number of (strictly) preperiodic points as
        \begin{align*}
            \sum_{k > 0} \#\Pre_k(P_d,\A^N(\F_p)) &=
            \sum_{e=0}^N (d^{N-e}-1)\binom{N}{e}(m^-)^{N-e}\left(\sum_{i=1}^{\nu-1}(d^{N-e})^k\right) \\
            &=\sum_{e=0}^N \binom{N}{e}(m^-)^{N-e}(d^{(N-e)\nu}-1)\\
            &=\sum_{e=0}^N \binom{N}{e} (p-1)^{N-e} - \binom{N}{e}(m^-)^{N-e}\\
            &= p^N - (m^-+1)^N,
        \end{align*}
        which is all points minus the periodic ones. Similarly, for projective space,
        \begin{align*}
            \sum_{k > 0} \#\Pre_k(P_d,\P^N(\F_p)) &=
            \sum_{D=0}^N\sum_{e=0}^D (k^{D-e}-1)\binom{D}{e}(m^-)^{D-e}\left(\sum_{k=1}^{\nu-1}(d^{D-e})^k\right) \\
            &=\sum_{e=0}^D \binom{D}{e}(m^-)^{D-e}(d^{(D-e)\nu}-1)\\
            &=\sum_{e=0}^D \binom{D}{e} (p-1)^{D-e} - \binom{D}{e}(m^-)^{D-e}\\
            &= \sum_{D=0}^N p^D - \sum_{D=0}^N(m+1)^D. \qedhere
        \end{align*}
    \end{proof}

    \begin{exmp}
        If $p-1 = \prod p_i$ = d, then the map has $2^{N+1}-1$ fixed points and all the other points are preperiodic leaves. This gives the largest number and the smallest (other than permutations) ratio of preperiodic leaves to periodic points.
    \end{exmp}

\section{Further Questions} \label{sect_further}
There are a number interesting polynomial maps outside the scope of this article that warrant further study. For example, the multivariate Chebyshev polynomials, e.g., $f(x,y) = (x^2-2y, y^2-2x)$, seem to have many more periodic points than any of the split polynomial maps. It would be interesting to establish a classification of all polynomial maps in higher dimensions for which this article is a first step.

Additionally, while dynamical zeta functions and periodic point counts have been studied in dimension 1 \cite{Bridy, Manes4}, little has been done in higher dimensions. It seems reasonable to expect that the results in this article could be extended to $\F_{p^k}$ (similar to \cite{Sha}) and applied to such problems.

\begin{code}
\begin{python}
def ord(x,p):
    return Zmod(p)(x).multiplicative_order()

def val(x,p):
    return x.valuation(p)

def t(x,p,d):
    return val(ord(x,p),d)

def t2(x,y,p):
    return max(t(x,p),t(y,p))

def c(x,p):
    l=ord(x,p)
    while 2.divides(l):
        l=l/2
    return ord(2,l)

def c2(x,y,p):
    return lcm(c(x,p),c(y,p))

def find_h(ordh,d):
    for n in range(1,d):
       if gcd(2,n)==1:
            if ord(2,n) == ordh:
                print n

def periodic_count(p,k,N,check=False,per=0):
    import itertools
    m=(p-1)/prod([(t**((p-1).ord(t))) for t in k.prime_divisors()])
    P=ProjectiveSpace(GF(p),N)
    H=End(P)
    Pd=H([z^k for z in P.gens()])
    if check:
        G=Pd.cyclegraph().all_simple_cycles()

    if per==0:
        #return all perioic points
        c=0
        for dim in range(N,-1,-1):
            for I in range(dim,-1,-1):
                VV= [ZZ(m).divisors() for i in range(I)]
                if I==0 and per==1:
                    c+=1#origin is fixed
                if I==1:
                    V=VV[0]
                else:
                    V = list(itertools.product(*VV))
                for v in V:
                    if I==1:
                        c+=binomial(dim,I)*euler_phi(v)
                    else:
                        c+= binomial(dim,I)*prod([euler_phi(w) for w in v])
        if check:
            s=sum([len(t)-1 for t in G])
            return (c,s)
        return c
    else:
        #count all of given period
        c=0
        for dim in range(N,-1,-1):
            for I in range(dim,-1,-1):
                VV= [ZZ(m).divisors() for i in range(I)]
                if I==0 and per==1:
                    c+=1#origin is fixed
                if I==1:
                    V=VV[0]
                else:
                    V = list(itertools.product(*VV))
                for v in V:
                    if I==0:
                        pers = [0]
                    elif I ==1:
                        pers = [ord(k,v)]
                    else:
                        pers = [ord(k,w) for w in v]
                    if lcm(pers)==per:
                        if I==1:
                            c+=binomial(dim,I)*euler_phi(v)
                        else:
                            c+= binomial(dim,I)*prod([euler_phi(w) for w in v])
        if check:
            s=0
            for t in G:
                if (len(t)-1)==per:
                    s+=per
            return(c,s)
        return c
#----------------

p=37
A.<z>=AffineSpace(GF(p),1)
A2.<x,y>=AffineSpace(GF(p),2)
HA=End(A)
H2=End(A2)
f=HA([z^2])
f2 = H2([x^2,y^2])

#dim 1 brute force
per1 = [0 for i in range(20)]
for v in A:
    ta,pe=v.orbit_structure(f)
    if ta ==0:
        per1[pe]+=1
per1

#dim2 from dim 1
per2=[0 for i in range(len(per1))]
for i in range(len(per1)):
    if per1[i]!=0:
        for d in ZZ(i).divisors():
            for e in ZZ(i).divisors():
                if lcm(d,e) == i:
                    per2[i]+=per1[d]*per1[e]
per2

#brute force check
per = [0 for i in range(20)]
for v in A2:
    ta,pe=v.orbit_structure(f2)
    if ta ==0:
        per[pe]+=1
per


per2=[0 for i in range(len(per1))]
#all nonzero
m = (p-1)/(2**(val(p-1,2)))
for d in ZZ(m).divisors():
    pd = ord(2,d) #period of each x
    wd = euler_phi(ZZ(d))/pd #number of cycles
    for e in ZZ(m).divisors():
        pe = ord(2,e)  #period of y
        we = euler_phi(ZZ(e))/pe #number of cycles
        n = lcm(pd,pe) #overall period
        per2[n] += wd*we*(gcd(pe,pd))*n
print "A:",per2
#(x,0) and (0,y)
for d in ZZ(m).divisors():
    pd = ord(2,d) #period
    wd = euler_phi(ZZ(d))/pd #num cycles
    per2[pd] += 2*(wd*pd)
per2[1]+=1 #(0,0)
print "B:",per2

#-----------------
per = [0 for i in range(20)]
tail = [0 for i in range(20)]
for v in A2:
    ta,pe=v.orbit_structure(f2)
    tail[ta] += 1
    if ta ==0:
        per[pe]+=1
print "per:",per
print "tail:",tail

#total
tot=0
totk=[0 for i in range(20)]
for i in range(1,k+1):
    tot+=3/4*m^2*4^i + 1/2*2*m*2^i
    totk[i]+=3/4*(m^2*4**(i)) + 1/2*(2*m*2**(i))
print "tot:",tot
print "totk:",totk
\end{python}
\end{code}

\bibliography{masterlist}
\bibliographystyle{plain}

\end{document}